\documentclass[12pt,a4paper,final]{amsart}

\usepackage{amssymb,amsmath,amsthm}
\usepackage{color}
\usepackage{rotating}
\usepackage{fullpage}
\usepackage{comment}
\DeclareMathOperator{\inte}{int}

\newtheorem{theorem}{Theorem}
\newtheorem{corollary}[theorem]{Corollary}
\theoremstyle{definition}
\newtheorem{remark}[theorem]{Remark}
\newtheorem{definition}{Definition}

\newcommand{\R}{{\mathbb R}}

\title{A computer-assisted proof of symbolic dynamics\\ in Hyperion's inner rotation model}

\author{Anna Gierzkiewicz}
\email{anna.gierzkiewicz@urk.edu.pl}
\address{Department of Applied Mathematics, University of Agriculture in Krak\'ow,
ul. Balicka 253c,
30--198 Krak\'ow, Poland
}

\author{Piotr Zgliczy\'nski}
\email{umzglicz@cyf-kr.edu.pl}
\address{Institute of Computer Science, Jagiellonian University,
ul. \L ojasiewicza 6, 30-348 Krak\'ow, Poland
}
\thanks{Work of A.G. and P.Z. was supported by National Science Center
(NCN) of Poland under project No. 2015/19/B/ST1/01454.}

\begin{document}

\begin{abstract}

The rotation of Hyperion is often modelled by equations of motion of an ellipsoidal satellite. The model is expected to be chaotic for large range of parameters. The paper contains a rigorous computer-assisted proof of the existence of symbolic dynamics in its dynamics by the use of CAPD C++ library.


\end{abstract}

\maketitle


\section{Introduction}


A motivation of our study was the broadly known example of chaotic motion in the Solar System, \textit{i.e.} the tumbling of Hyperion, one of Saturn's moons. The shape of Hyperion significantly differs from spherical: it is roughly an ellipsoid with the rate of principal moments of inertia $\frac{\Theta_2-\Theta_1}{\Theta_3}\approx 0.26$, where $\Theta_3>\Theta_2>\Theta_1$. Hyperion's Keplerian elliptic orbit (long semi-axis $a=1, 500, 933$ km, period $T=21.276$ d) is non-circular with eccentricity $e\approx 0.1$ \cite{K}.

The analysis of \textit{Voyager} and \textit{Voyager 2} observational data could not fit Hyperion's rotation into any certain period, which was the base of suspection that its tumbling may be chaotic. The natural attempt \cite{W} to explain this phenomenon was to compare it to a classical \cite{Danby} model of inner rotation of an oblate satellite, see \cite[Ex. 27.5]{Greiner}. The model assumes that the satellite is of ellipsoidal shape and orbits a massive distant body in a Keplerian ellipse orbit with significant eccentricity. It also states that the longest (or the shortest) ellipsoid's axis is always perpendicular to the orbit's plane, which is crucial to simplicity of the model: it implies one axis of rotation only. This last assumption is justified as the basic analysis of Euler equations of the rigid body motion shows that this `normal' state is stable \cite{Danby}.

The model does not, unfortunately, fit sufficiently: the key assumption on perpendicularity of the rotation axis to the orbit plane is not true in Hyperion's case. Nevertheless, the model is interesting and applicable in many other cases, like Moon's or Mercury's libration. Proving the existence of chaos in it can also be helpful in the more general modelling of Hyperion's motion. This is the reason why in this paper we fix the parameters to Hyperion's case: $\omega = 0. 89 \pm 0. 22$ and $e=0.1042$ \cite{K}.

The model with the above parameters was, as mentioned, explored in some articles, such as \cite{B,K,W,Tarnopolski2015}. The statement of chaotic rotation is based there on the picture of Poincar\'e section $S:\{f=0\}$, which visibly contains a large chaotic region (see also Fig \ref{fig:P}). The Lyapunov Characteristic Exponents were also numerically calculated. The rigorous proof of chaoticity would set the mathematical ground to their theses and also can present an elegant application for the topological methods combined with rigorous numerics\cite{CAPD}. In this paper we understand the existence of chaos in a dynamical system as the semiconjugacy of its dynamics onto the shift dynamics on the space of bi-infinite sequences of two symbols. Such a phenomenon is known in literature as \emph{symbolic dynamics} \cite{Morse, Moser}.

As many proofs based on the interval arithmetic, our investigation is valid for a small interval of parameters, containing the values for Hyperion. The methods can be easily applied to other (but precise) values of $\omega$, $e$. The more general question for our future work is to explore the relation between the value of parameters and the size and structure of the chaotic region on the section $S$, which could be applied in modelling rotation of other objects of the universe.

Another question is the way of generalizing the model itself. If it cannot describe the tumbling of Hyperion, then the complete set of Euler equations could be investigated. This extends the phase space from $3$ to $7$ dimensions and requires the use of more complex methods. The other way to make the model more accurate is to consider the impact of Titan on the Saturn--Hyperion system as a correction in the rotation equation \cite{Tarnopolski2016} or as a change of parameter $e$ to become an independent variable.

The paper is organized as follows: in Section \ref{sec:model} we present the model, the system of ODEs and its basic properties. Sections \ref{sec:tools} and \ref{sec:chaos} contain description of main topological tools used in our work. The last Section \ref{sec:results} presents
our results for the symbolic dynamics in our model.


\section{The model}\label{sec:model}
\subsection{The equations}

We shortly recollect the derivation of the model \cite[Eq. 14.3.1]{Danby}.

	\begin{figure}[h]
	\includegraphics[height=5cm]{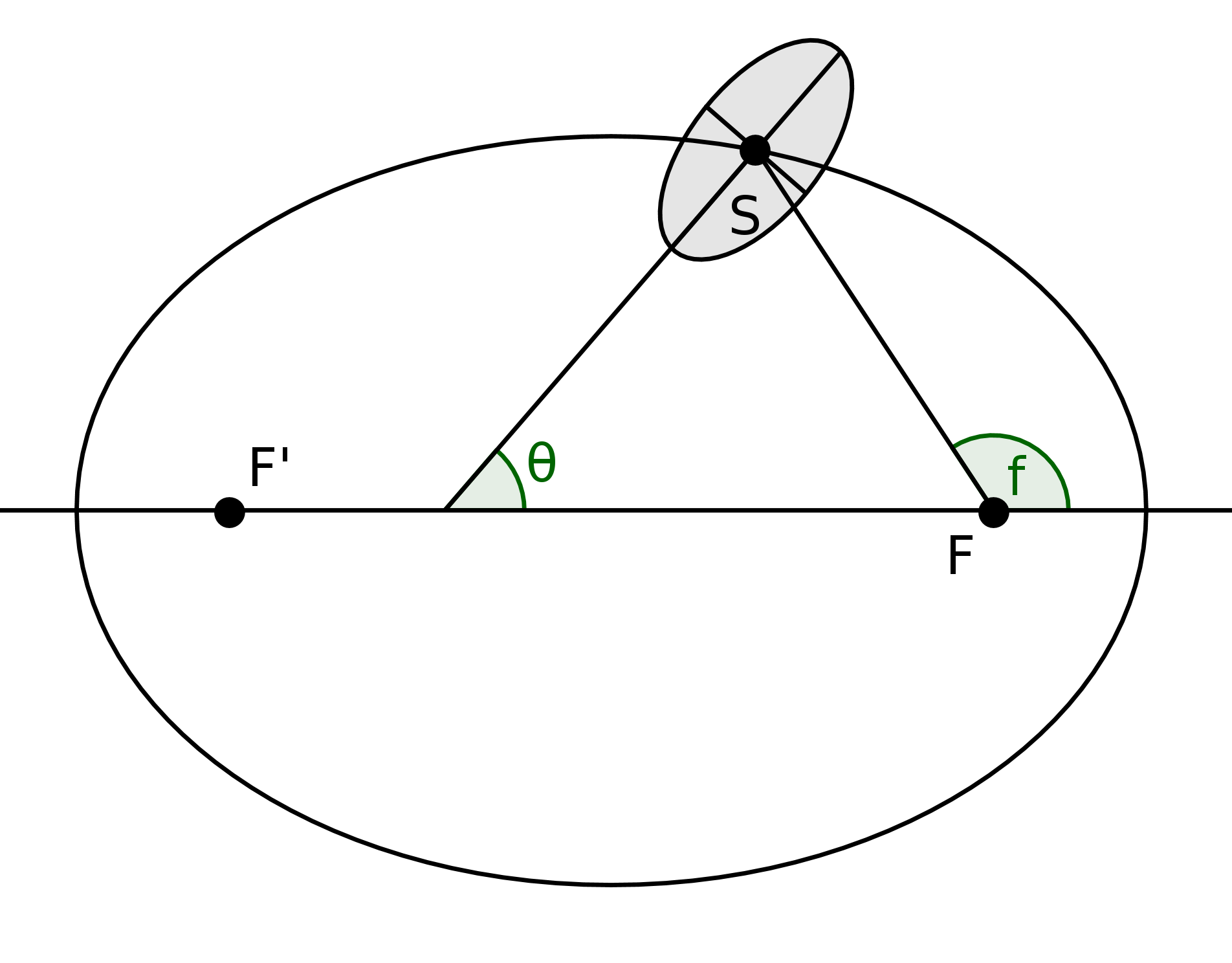}
	\caption{\label{fig:satellite} Illustration of the model.}
	\end{figure}

An ellipsoidal satellite $S$ orbits a massive body $F$ in a Kepler ellipse orbit. Therefore its true anomaly $f$ fulfils the equation
\begin{equation}\label{f}
f'=\frac{(1+e \cos f)^2}{(1-e^2)^{3/2}}.
\end{equation}
The equation (\ref{f}) has a symmetry: if $t\mapsto f(t)$ is a solution, then $t\mapsto -f(-t)$ is a solution. Also, solutions are strictly increasing.

The shortest axis of the satellite is perpendicular to the plane of the orbit. The inner rotation is expressed by the angle $\theta$ (see Fig. \ref{fig:satellite}) between the longest axis of $S$ and the long axis of the orbit. Then $\theta$ fulfils the second-order ordinary differential equation
\cite[Eq. 27.97]{Greiner}
\begin{equation}\label{theta}
\theta''=-\frac{\omega^2}{2r^3} \sin 2(\theta-f)
\text{,\qquad where }
\qquad
r=\frac{1-e^2}{1+e\cos f}.
\end{equation}
The parameter
\begin{equation}
\omega^2 = 3a^3\left(\frac{2\pi}{T}\right)^2\frac{\Theta_2-\Theta_1}{\Theta_3}\in [0,1]
\end{equation}
may be related to normalized oblateness of the satellite.


\subsection{The dynamical system}
In general, Eqs. (\ref{f}) and (\ref{theta}) induce a three-dimensional dynamical system
\begin{equation}\label{rot}
\begin{cases}
\theta'=\phi\\
\phi'=-\frac{\omega^2}{2r^3} \sin 2(\theta-f)\\
f'=\frac{(1+e \cos f)^2}{(1-e^2)^{3/2}}
\end{cases}
\end{equation}
with parameters $e$, $\omega^2$.
The inner rotation angle $\theta \in [0,\pi]$ and $f\in [0,2\pi]$, so the phase space for the system (\ref{rot}) is
$(\theta,\phi,f)\in \R_{/\pi\mathbb{Z}}\times \R \times \R_{/2\pi\mathbb{Z}}$.


\subsection{Poincar\'e map}

We study the Poincar\'e map $P$ of the system (\ref{rot}) on the 2-dim section $S:\{f=0\}$, \textit{i.e.} the map
\[
P(\theta,\phi) = \Phi\left(T(\theta,\phi), (\theta,\phi,f=0)\right)\text{,}
\]
where $\Phi$ is the dynamical system induced by (\ref{rot}) and $T=T(\theta,\phi)$ is a first recurrence time.
Note that the domain of so defined map is $\operatorname{Dom}_P = \R_{/\pi\mathbb{Z}}\times \R$, because $f$ is strictly increasing and of bounded variation.

The main fragment of the Poincar\'e section $S$ with twelve orbits marked in different colours is depicted on Fig. \ref{fig:P}.
	\begin{figure}[h]
	\includegraphics[height=7cm]{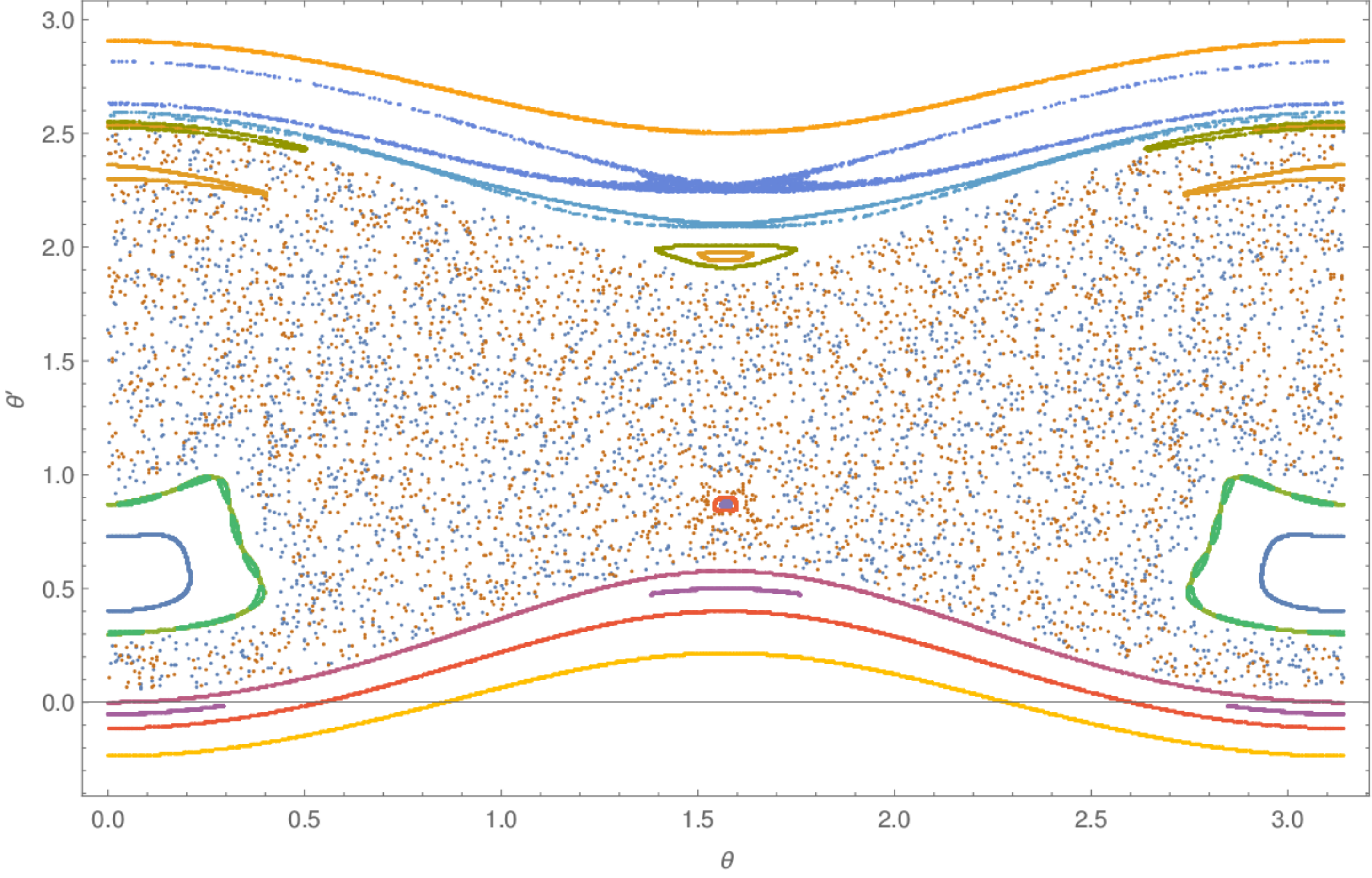}	
	\caption{\label{fig:P} Poincar\'e section $S: \{f=0\}$, $e=0.1$, $\omega^2=0.79$}
	\end{figure}

What can be immediately noticed is a probable symmetry of the map $P$, because the section seems to have reflection symmetries with respect to $\theta=0=\pi$ and $\theta=\frac{\pi}{2}$ lines. Indeed, setting $t\mapsto -t$ to the equations we notice that if $(\theta(t), \phi(t), f(t))$ is a solution, then so is $(-\theta(-t), \phi(-t), -f(-t))$ and consequently
\[
\pi_\theta P(\theta,\phi) = -\pi_\theta P^{-1}(-\theta,\phi)
\text{, \qquad }
\pi_\phi P(\theta,\phi) = \pi_\phi P^{-1}(-\theta,\phi).
\]
Then, using the periodicity of the phase space: $-\theta = 0-\theta = \pi-\theta$, we explain the two symmetries of the Poincar\'e section.

For further consideration, we denote the $\{\theta=\frac{\pi}{2}\}$-hyperplane reflectional time-reversing symmetry of the extended phase space by $R$, so
\[
R\left(\theta,\phi,f,t\right) = \left(\pi-\theta,\phi,-f,-t\right).
\]
We will also, if it is understandable, denote by $R$ its restriction to $S$: $R(\theta,\phi) = (\pi-\theta,\phi)$.

\subsection{Chaos in the system (\ref{rot})}
The other natural observation based on Fig. \ref{fig:P} is a large region of probable chaos for (more or less) $0 < \phi=\theta' < 2.0$, with some elliptic islands. The chaotic behaviour in terms of stability and tidal evolution was studied by Wisdom \textit{et al.} in \cite{W}. The Lyapunov Characteristic Exponents (LCE) occured to be positive, which was the reason to treat Hyperion's motion as chaotic in the subsequent literature. LCE of the system with wide range of $e$, $\omega^2$ were also explored in \cite{Tarnopolski2015}.


\section{Periodic orbits via topological covering}\label{sec:tools}

Topological tools that we used in detecting periodic orbits for (\ref{rot}) were introduced in details in \cite{WZ1,WZ2}. Here we recollect them shortly and present some intuition.

\subsection{H-sets}
The basic notion is
\begin{definition}[\cite{WZ1}, Def. 3.1]
\emph{An h-set} is a quadruple $N=(|N|, u(n), s(N), c_N)$, where $|N|$ is a compact subset of $\mathbb{R}^n$, which we will call \emph{a support of a h-set} (or simply an h-set) and
\begin{enumerate}
\item two numbers $u(N)$, $s(N) \in \mathbb{N}\cup\{0\}$ complement the dimension of space:
\[u(N) + s(N)=n\text{;}\]
we will call them the \emph{exit} and \emph{entry dimension}, respectively;
\item the homeomorphism $c_N : \mathbb{R}^n\to\mathbb{R}^n=\mathbb{R}^{u(N)}\times \mathbb{R}^{s(N)}$ is such that
\[
c_N(|N|)=\overline{\mathbb{B}_{u(N)}}\times\overline{\mathbb{B}_{s(N)}}\text{,}
\]
where $\overline{\mathbb{B}_{k}}$ denotes a closed unit ball of dimension $k$.
\end{enumerate}

\end{definition}
 We set also some useful notions:
\begin{align*}
\dim N &= n\text{,}\\
N_c &= \overline{\mathbb{B}_{u(N)}}\times\overline{\mathbb{B}_{s(N)}}\text{,}\\
N_c^- &= \partial\mathbb{B}_{u(N)}\times\overline{\mathbb{B}_{s(N)}}\text{,}\\
N_c^+ &= \overline{\mathbb{B}_{u(N)}}\times\partial\mathbb{B}_{s(N)}\text{,}\\
N^- &= c_N^{-1}(N_c^-)\text{,} \qquad N^+ = c_N^{-1}(N_c^+)\text{.}
\end{align*}
As one can notice, the notions with the subscript $_c$ refer to the `straight' coordinate system in the image of $c_N$. The last two sets $N^-$ and $N^+$ defined above are often called the \emph{exit set} and the
\emph{entrance set}, respectively.

Therefore, we can assume that an h-set is a product of two unitary balls moved to some coordinate system with the exit and entrance sets distinguished.

\subsection{Covering and back-covering}
We define the notion of topological covering:

\begin{definition}[\cite{WZ1}, Def. 3.4, simplified]
Let $f:|N|\to \mathbb{R}^n$ be a continuous map and two h-sets $M$, $N$ are such that
$u(M)=u(N)=u$ and $s(M)=s(N) = s$. Denote $f_c=c_N \circ f \circ c_M^{-1} : M_c \to \mathbb{R}^u\times \mathbb{R}^s$.
We say that that $M$ \emph{$f$-covers} the h-set $N$, if
\begin{enumerate}
\item there exists a continuous homotopy $h:[0,1]\times M_c \to \mathbb{R}^u\times \mathbb{R}^s$, such that:
\begin{align*}
h_0 &= f_c \text{,} &\\
h([0,1],M_c^-)\cap N_c &= \varnothing &\textit{(the exit condition),}\\
h([0,1],M_c)\cap N_c^+ &= \varnothing &\textit{(the entry condition).}\\
\end{align*}
\item If $u>0$, then there exists a linear map $A:\mathbb{R}^u \to \mathbb{R}^u$ such that
\begin{align*}
h_1(x,y) &= (A(x),0) \quad \text{ for } x\in\overline{\mathbb{B}_u} \text{ and } y\in\overline{\mathbb{B}_s} \text{,}\\
A(\partial \mathbb{B}_u) &\subset \mathbb{R}^u\setminus\mathbb{B}_u.
\end{align*}
\end{enumerate}
\end{definition}
If $M$ $f$-covers $N$, we simply denote it by $M \stackrel{f}{\Longrightarrow}N$. See Fig. \ref{fig:covering} for an illustration of covering in some low-dimensional cases.

\begin{figure}[h]
	\includegraphics[height=3cm]{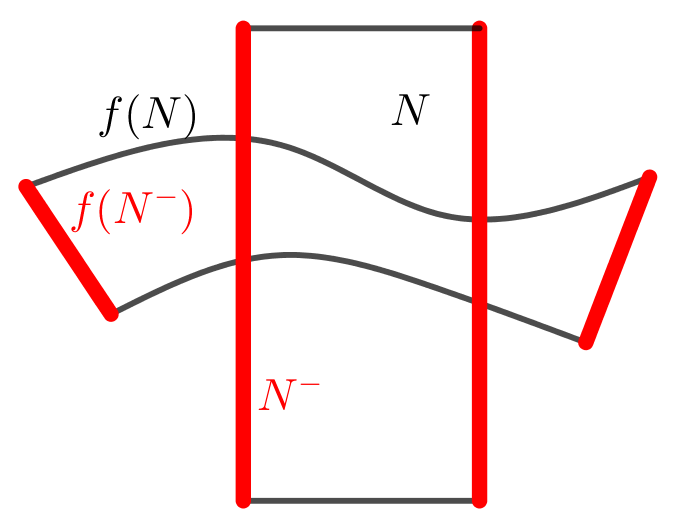}
	\hfil
	\includegraphics[height=3cm]{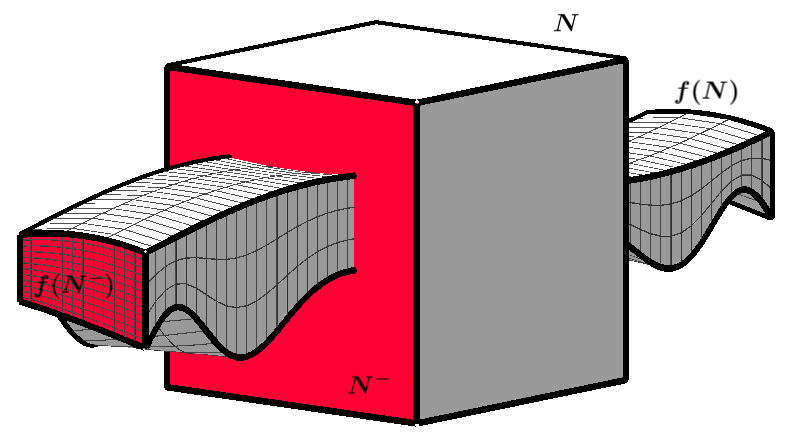}
	\hfil 	
	\includegraphics[height=3cm]{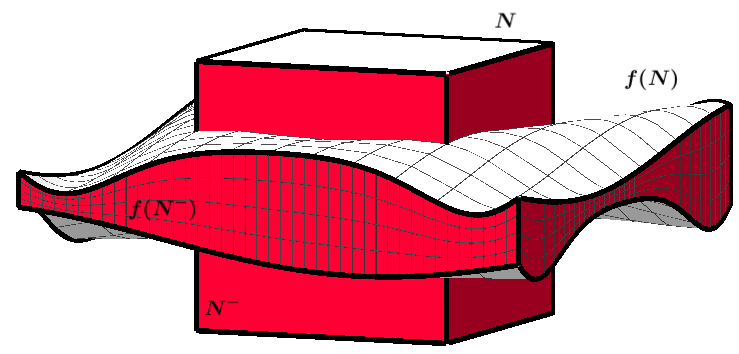}
	\caption{\label{fig:covering}Examples of topological self-covering $N \stackrel{f}{\Longrightarrow}N$ in $\mathbb{R}^2$ (left) and $\mathbb{R}^3$: with one exit direction (middle) and two exit directions (right). The exit sets and their images are marked in red.}
\end{figure}

Sometimes it is more convenient to use the backward covering, which one may understand as the covering backwards in time.

\begin{definition}
If $N$ is an h-set, then we define an h-set $N^T$ as
\begin{itemize}
	\item $|N^T| = |N|$;
	\item $u(N^T) = s(N)$ and $s(N^T) = u(N)$;
	\item $c_{N^T}:\mathbb{R}^n \ni x \longmapsto j(c_N(x)) \in  \mathbb{R}^{u(N^T)}\times\mathbb{R}^{s(N^T)}=\mathbb{R}^n$,\\
	where
	$j:\mathbb{R}^{u(N^T)}\times\mathbb{R}^{s(N^T)}\ni (p,q) \longmapsto (q,p) \in \mathbb{R}^{s(N^T)}\times\mathbb{R}^{u(N^T)}$.
\end{itemize}
\end{definition}
As we can see, the h-set $N^T$ is just the h-set $N$ with the entrance and exit sets swapped.
\begin{definition}
Let $M$, $N$ be two h-sets such that $u(M)=u(N)=u$ and $s(M)=s(N) = s$. Let $f:\operatorname{Dom}_f\subset\R^n \to \R^n$ be such that $f^{-1}:|N|\to \R^n$ is well-defined and continuous. Then we say that $M$ \emph{back-covers} $N$, and denote by $M\stackrel{f}{\Longleftarrow}N$, iff $N^T\stackrel{f^{-1}}{\Longrightarrow} M^T$.\\
If either $M\stackrel{f}{\Longleftarrow}N$ or $M\stackrel{f}{\Longrightarrow} N$, then we will write $M\stackrel{f}{\Longleftrightarrow} N$.
\end{definition}

In general, if $N_0 \stackrel{f}{\Longrightarrow}N_1$ and $N_1 \stackrel{f}{\Longrightarrow}N_2$, then not necessarily $N_0 \stackrel{f^2}{\Longrightarrow}N_2$, but covering has the property of tracking orbits. The basic application of topological covering is the following theorem, stating the existence of a periodic orbit related to a sequence of coverings.

\begin{theorem}[\cite{WZ1}, Theorem 3.6, simplified]\label{th:periodic}
Suppose there exists a sequence of h-sets $N_0$, \ldots $N_n=N_0$, such that
\[
N_0 \stackrel{f}{\Longleftrightarrow} N_1 \stackrel{f}{\Longleftrightarrow}\ldots \stackrel{f}{\Longleftrightarrow} N_n = N_0\text{,}
\]
then there exists a point $x\in \inte |N_0|$, such that $f^k(x) \in \inte|N_{k}|$ for $k=0,1,\ldots,n$ and $f^n(x)=x$.
\end{theorem}

In particular, if $N_0\stackrel{f}{\Longrightarrow} N_0$, then in $N_0$ we have a stationary point for the map $f$. Note also that if the map is a Poincar\'e map $P$, then a stationary point for $P$ or $P^k$ lies on a periodic orbit for the dynamical system.
		
\subsection{Periodic orbits in Hyperion's rotation}


Using Theorem \ref{th:periodic}, we find some stationary points for $P$, \textit{i.e.} periodic orbits for the system (\ref{rot}). Their existence is proven rigorously via the interval Newton method \cite{M,N} implemented in C++ language with the use of CAPD library \cite{CAPD}. Using this method, one can also estimate the eigenvalues of the derivative in the stationary point, so it is possible to prove rigorously whether the points are hyperbolic. Those periodic points are depicted on the Poincar\'e map $P$ on Fig. \ref{fig:PPoints}. Three of them, denoted as $P_1$ $P_2$ and $P_3$, will be important in further consideration.
\begin{figure}[h]
\includegraphics[height=7cm]{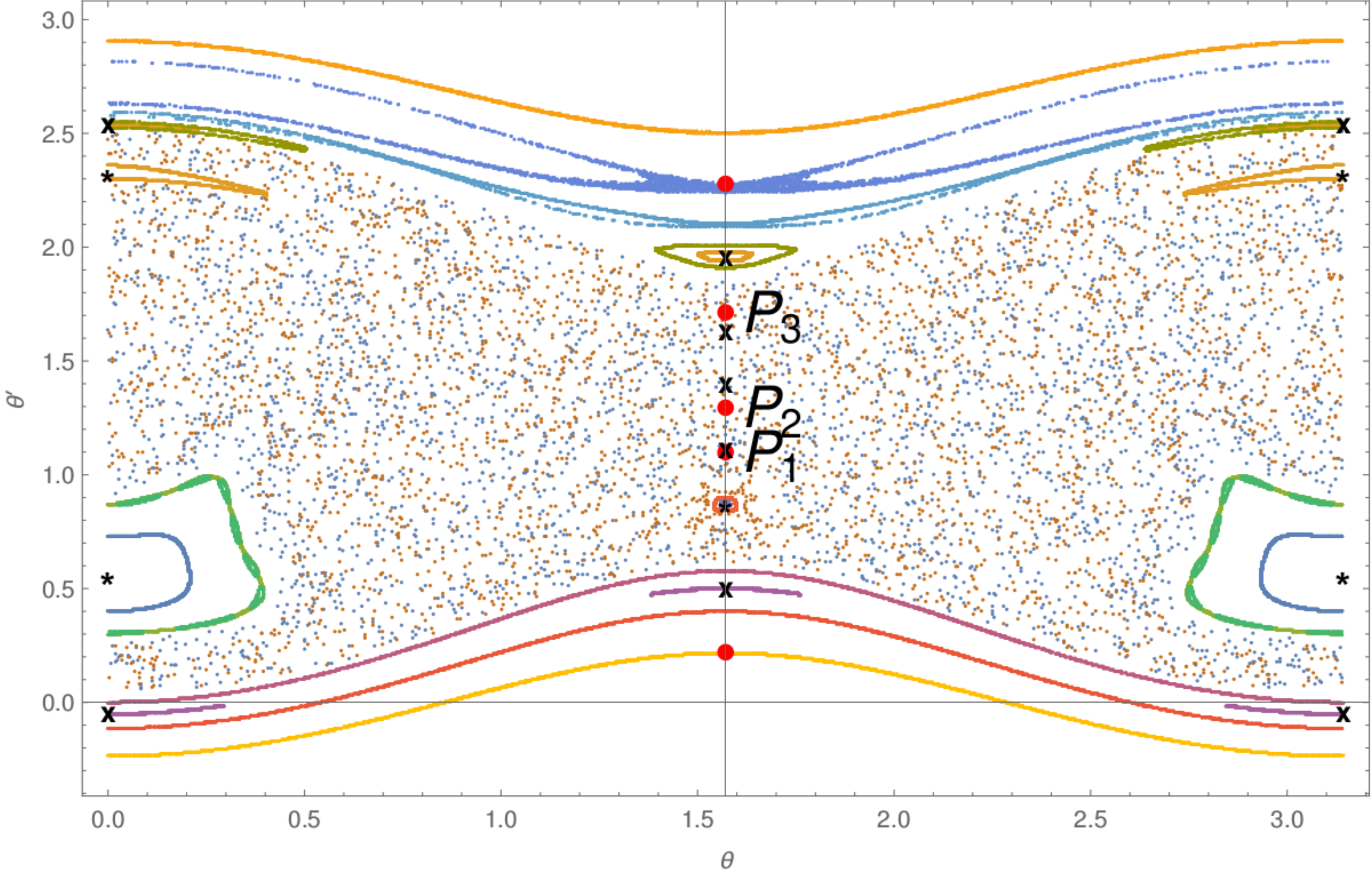}
\caption{\label{fig:PPoints} Periodic points of $P$, found via interval Newton method. The points marked by red dots are hyperbolic, black asterisks `$\star$' are elliptic. The points marked by black `{\bf\sf x}' are the stationary points for $P^2$. }
\end{figure}

The list of small intervals on the $\{\theta=\frac{\pi}{2}\}$ axis containing  $P_1$, $P_2$, $P_3$ is in the Table \ref{table:stationary_points}.
From now on, we will denote by $P_1$, $P_2$, $P_3$ the stationary points as well as the small sets containing them, described in this Table.

\begin{table}[h]
\begin{tabular}{|c|c|}
\hline
point & interval on $\{\theta=\frac{\pi}{2}\}$ \\
\hline\hline
$P_1\in$ & $[1.098956671156713, 1.098956671156731] = 1.0989566711567_{13}^{31}$
\\
\hline
$P_2\in$ & $[1.294511656257196, 1.294511656257254] = 1.294511656257_{196}^{254}$
\\
\hline
$P_3\in$ & $[1.712042516112098, 1.712042516112223] = 1.712042516112_{098}^{223}$
\\
\hline
\end{tabular}

\caption{\label{table:stationary_points}Localization of three stationary points of $P$, found via interval Newton method.}
\end{table}

The hyperbolic points $P_1$, $P_2$, $P_3$, presented on Fig. \ref{fig:PPoints}, can be also detected up to a small neighbourhood using covering relations. It is sufficient to find a self-covering compact set. This method, however, does not prove neither the uniqueness of the stationary point inside the set nor its hyperbolicity. The examples of the self-covering sets for $P_1$, $P_2$, $P_3$ are presented on Fig. \ref{fig:PiCov}.

\begin{figure}[h]
	\includegraphics[width=0.27\textwidth]{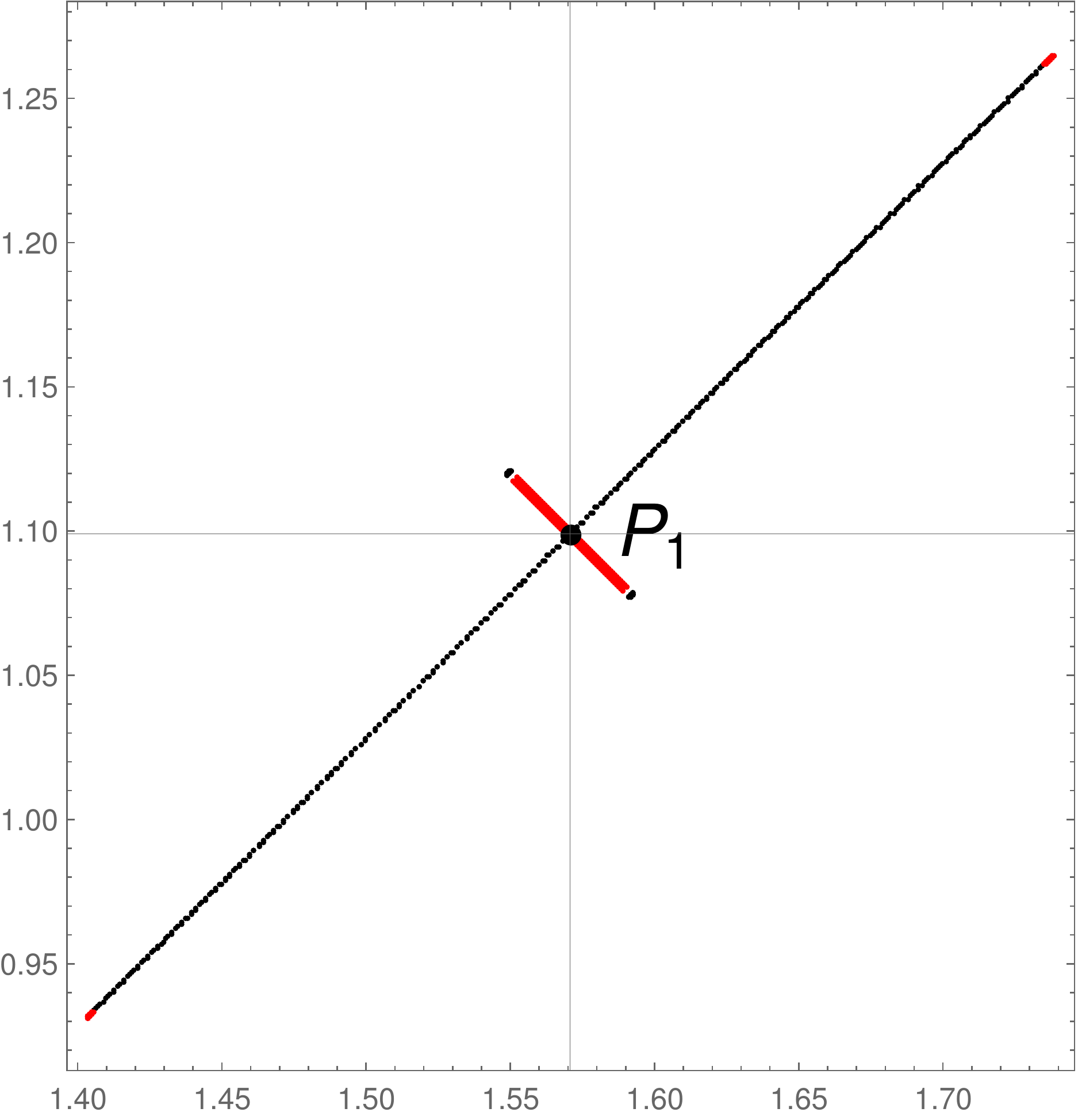}
	\qquad
	\includegraphics[width=0.27\textwidth]{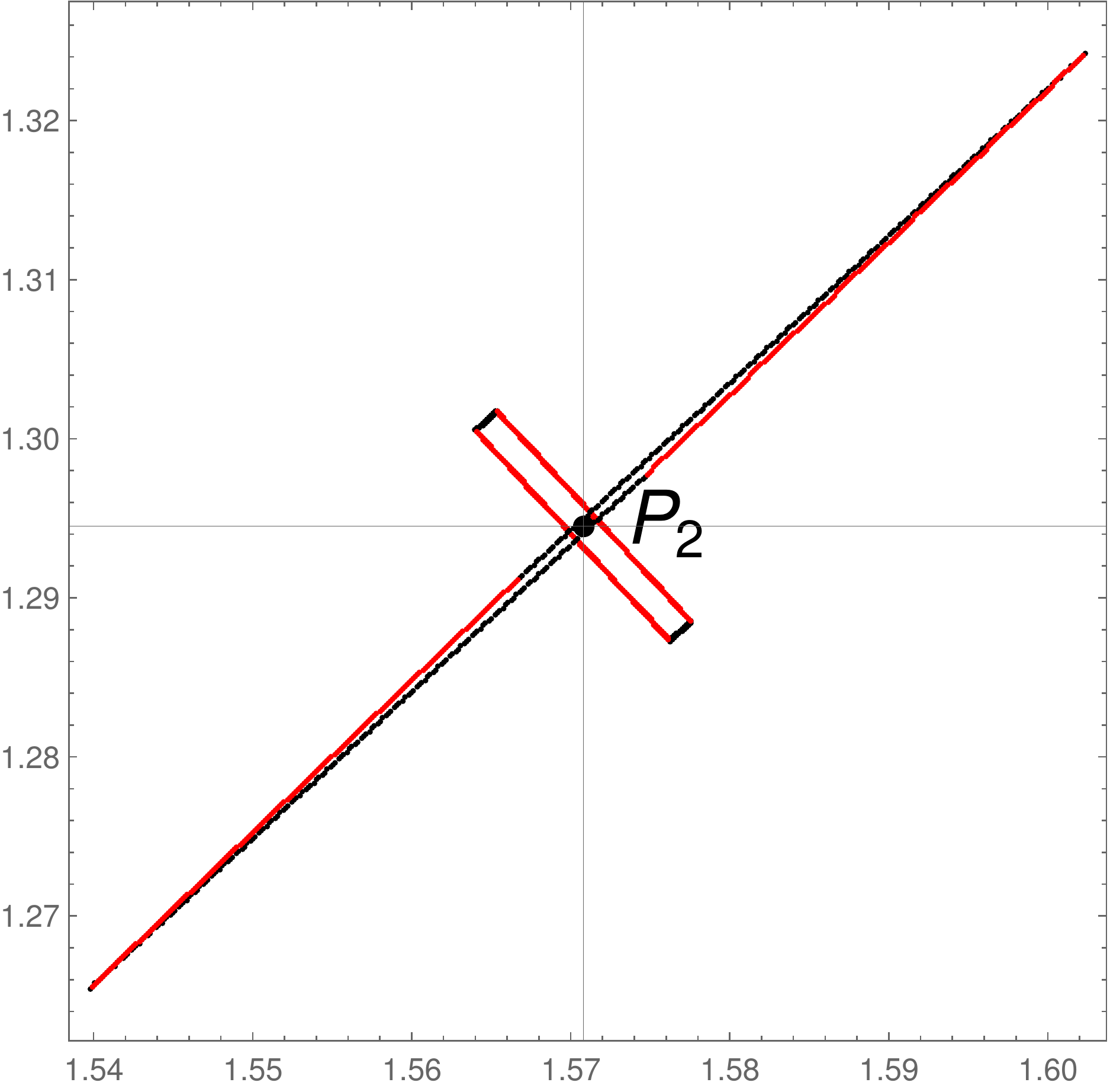}
	\qquad
	\includegraphics[width=0.27\textwidth]{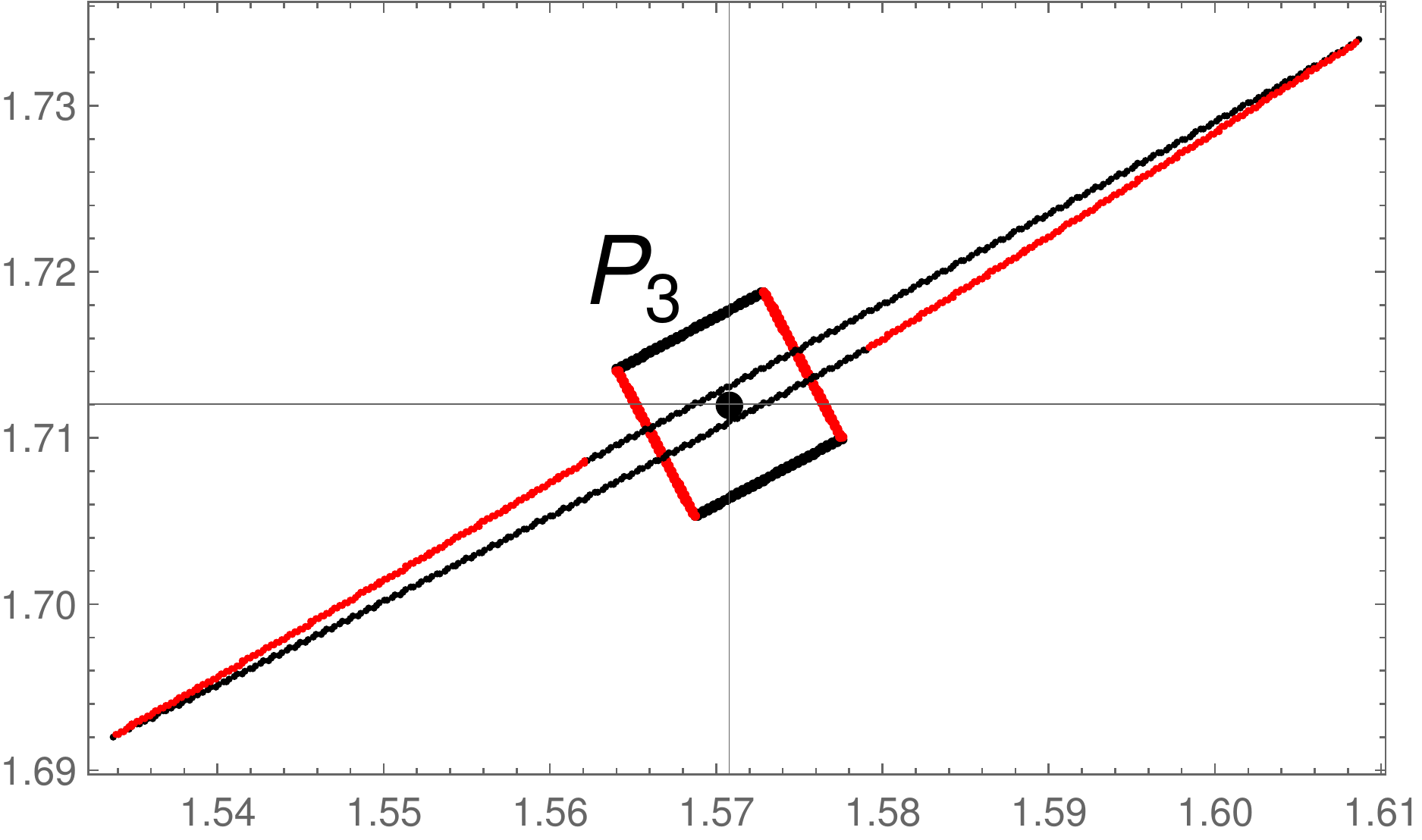}
	\caption{\label{fig:PiCov} Self-covering sets (the rectangles) proving existence of stationary points $P_1$, $P_2$, $P_3$. The exit sets and their images are marked in red.}
\end{figure}

\section{Symbolic dynamics detecting via covering}\label{sec:chaos}

For the next study we assume that the continuous map $f : \mathbb{R}^2 \to \mathbb{R}^2$ and all h-sets $N_i$ contained in $\mathbb{R}^2$ have entry and exit dimensions equal to 1, that is $s(N_i)=u(N_i)=1$.
Following the above assumptions the notions related to a h-set get simpler, because the balls are just the closed intervals: $\overline{\mathbb{B}_{u(N)}} = \overline{\mathbb{B}_{s(N)}}=[-1,1]$:
\begin{itemize}
\item $N_c = [-1,1]^2$,
\item $N_c^- = \{-1,1\}\times[-1,1]$,\quad $N_c^+ = [-1,1]\times\{-1,1\}$,
\item each $N^-$ and $N^+$ is topologically a sum of two disjoint intervals.
\end{itemize}

\subsection{Symbolic dynamics for a 2-dimensional map}

A model chaotic behaviour for our purposes is the shift map on the set of bi-infinite sequences of two symbols, that is, the space $\Sigma_2=\{0,1\}^{\mathbb{Z}}$ as a compact metric space with the metric
\[
\text{for }c=\{ c_n\}_{n\in\mathbb{Z}}\text{, }c'=\{ c'_n\}_{n\in\mathbb{Z}}\text{, }\qquad
\operatorname{dist}\left(c,c'\right) = \sum_{n=-\infty}^{+\infty} \frac{|c_n-c'_n|}{2^{|n|}}\text{,}
\]
which induces the product topology. The shift map $\sigma:\Sigma_2\to \Sigma_2$, given by
$$
 (\sigma (c))_n =c_{n+1}\text{,}
$$
is a homeomorphism of $\Sigma_2$ with well-known chaotic properties like the existence of a dense orbit, existence of orbit of any given period or that the set of periodic orbits is dense in the whole space. 

In our study, by the chaotic behaviour of a dynamical system we understand the existence of a compact set $I$ invariant for the Poincar\'e map $P$ (or sometimes its higher iteration) and a continuous surjection $g:I\to \Sigma_2$ such that $P|_I$ is semi-conjugated to $\sigma$, that is:
\[
g\circ P|_I = \sigma \circ g.
\]
Then one may say that $P$ admits on $I$ at least as rich dynamics as $\sigma$ on $\Sigma_2$.
The system $(\Sigma_2,\sigma)$ or any system (semi-)conjugated to it is sometimes described in literature as \emph{symbolic dynamics}.

\subsection{Topological horseshoe}

A simple example of symbolic dynamics semi-conjugated to $\sigma$ is a horseshoe:

\begin{definition}
Let $N_0$, $N_1 \subset \mathbb{R}^2$ be two disjoint h-sets. We say that a continuous map $f:\mathbb{R}^2 \to \mathbb{R}^2$ is a \emph{topological horseshoe} for $N_0$, $N_1$ if (see Fig. \ref{fig:horseshoe})
\begin{equation}
\begin{array}{cc}
N_0 \overset{f}{\Longrightarrow}N_0\text{,} \quad
&N_0 \overset{f}{\Longrightarrow}N_1\text{,}
\\
N_1 \overset{f}{\Longrightarrow}N_0 \text{,} \quad
&N_1 \overset{f}{\Longrightarrow}N_1.
\end{array}
\end{equation}
\end{definition}

\begin{figure}[h]
\begin{minipage}{0.45\textwidth}
\begin{equation*}
\begin{array}{cc}
N_0 \overset{f}{\Longrightarrow}N_0\text{,} \quad
&N_0 \overset{f}{\Longrightarrow}N_1\text{,}
\\
N_1 \overset{f}{\Longrightarrow}N_0 \text{,} \quad
&N_1 \overset{f}{\Longrightarrow}N_1.
\end{array}
\end{equation*}
\end{minipage}
\begin{minipage}{0.45\textwidth}
\begin{center}
	\includegraphics[height=4cm]{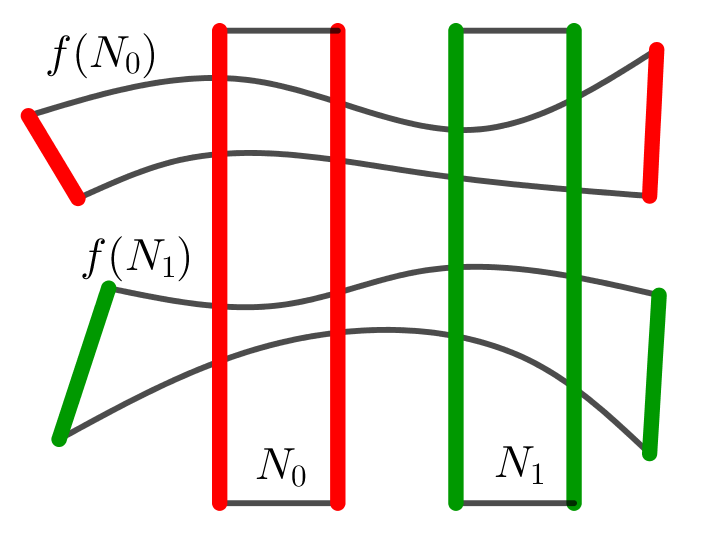}
\end{center}
\end{minipage}
	\caption{\label{fig:horseshoe}A topological horseshoe: each $N_{0,1}$ covers itself and the other set. The exit sets of $N_0$ and $N_1$ are marked in red and green, respectively.}
\end{figure}

It can be shown that for any topological horseshoe we obtain symbolic dynamics.

\begin{theorem}[\cite{GZ}, Theorem 18]
Let $f$ be a topological horseshoe for $N_0$, $N_1$. Denote by $I = \operatorname{Inv}(N_0\cup N_1)$ the invariant part of the set $N_0\cup N_1$ under $f$, and define a map $g: I \to \Sigma_2$ by
\[
g(x)_k = j \in \{0,1\} \quad \text{ iff } \quad f^k(x)\in N_j.
\]
Then $g$ is a surjection satisfying $g\circ f|_I = \sigma \circ g$ and therefore $f$ is semi-conjugated to the shift map $\sigma$ on $\Sigma_2$.
\end{theorem}

\begin{corollary}
Let $f$ be a topological horseshoe for $N_0$, $N_1$. Then it follows from Theorem \ref{th:periodic} that for any finite sequence of zeros and ones $(a_0, a_1,\ldots, a_{n-1})$, $a_i\in\{0,1\}$, there exists $x\in N_{a_0}$ such that
\[
f^i(x)\in\inte N_{a_i}
\qquad \text { and } \qquad
f^n(x)=x.
\]
\end{corollary}

\begin{remark}\label{rem:chain}
Note also that the following two chains of coverings or back-coverings:
\begin{gather*}
N_0 \overset{f}{\Longleftrightarrow} N_0 \overset{f}{\Longleftrightarrow} N_1 \overset{f}{\Longleftrightarrow} \dots \overset{f}{\Longleftrightarrow} N_k = M_0\text{,}
\\
M_0 \overset{f}{\Longleftrightarrow} M_0 \overset{f}{\Longleftrightarrow} M_1 \overset{f}{\Longleftrightarrow} \dots \overset{f}{\Longleftrightarrow} M_k = N_0
\end{gather*}
indicate symbolic dynamics for $f^k$ in the sets $N_0$ and $M_0$, with use of Theorem \ref{th:periodic}.
\end{remark}

\section{Symbolic dynamics in Hyperion's rotation system}\label{sec:results}

Numerical calculations suggest that the three hyperbolic points $P_1$, $P_2$ and $P_3$ detected above (see Table \ref{table:stationary_points}) have an interesting property: their stable and unstable manifolds intersect (for the same point and also pairwise), see Fig. \ref{fig:SUman}. The intersections seem to be transversal, which is a clue for searching for the symbolic dynamics in the Poincar\'e map $P$.

\begin{figure}[h]
	\includegraphics[height=6cm]{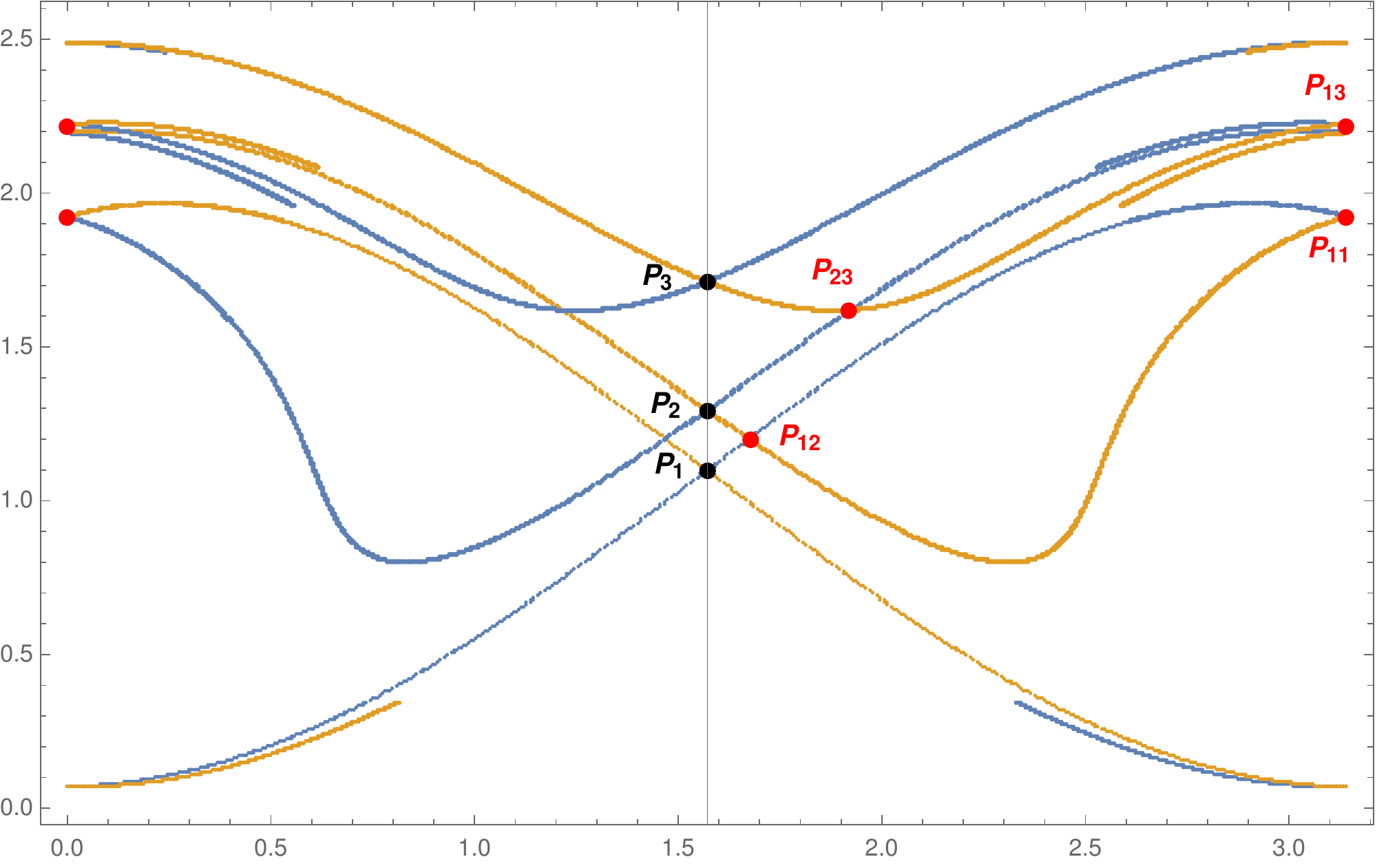}
	\caption{\label{fig:SUman}Fragments of stable (orange) and unstable (blue) manifolds of stationary points $P_1$, $P_2$, $P_3$. Some points of intersections of the manifolds are marked in red, namely $P_{11}$, $P_{12}$, $P_{13}$, $P_{23}$.}
\end{figure}

We found six topological horseshoes related to the intersections mentioned above. Below we present the h-sets and their covering relations. Similarly as in \cite{WZ1}, the h-sets are paralellograms of the form $N = p + A \cdot b$, which are the base cubes $b$ transformed to some affine coordinate system, where:
\begin{itemize}
\item $p$ are some small interval vectors containing base points (such as $P_1$, $P_2$ and $P_3$);
\item interval matrices $A$ are usually one of the eigenvectors matrices:
\begin{align*}
M_1 &= \begin{bmatrix}
0.70695646939_{59338}^{77127} & 0.70695646939_{59345}^{77133} \\ 0.7072570610_{281695}^{335063} & -0.7072570610_{335054}^{281689}
\end{bmatrix}
\text{,}
\\
M_2 &=\begin{bmatrix}
0.7344289378_{330728}^{407212} & 0.7344289378_{330732}^{407215}\\
0.6786855938_{153962}^{378076} & -0.6786855938_{378071}^{153957}
\end{bmatrix}
\text{,}
\\
M_3 &=\begin{bmatrix}
0.8837175133_{17777}^{293314} & 0.8837175133_{177765}^{293309}\\
0.4680206797_{026127}^{366671} & -0.4680206797_{366677}^{026135}
\end{bmatrix}
\text{,}
\end{align*}
or, in the case of connecting $P_3$ to itself, some matrices corrected to point directions of the dynamics on the stable or unstable manifolds.
\end{itemize}

\subsection{$P_1$ and $P_2$}

The simplest situation occurs in the case of the points $P_1$ and $P_2$, because we found a direct topological horseshoe for the first iteration of the Poincar\'e map $P$ (see Fig. \ref{fig:p12}).

\begin{figure}[h]
	\includegraphics[height=7cm]{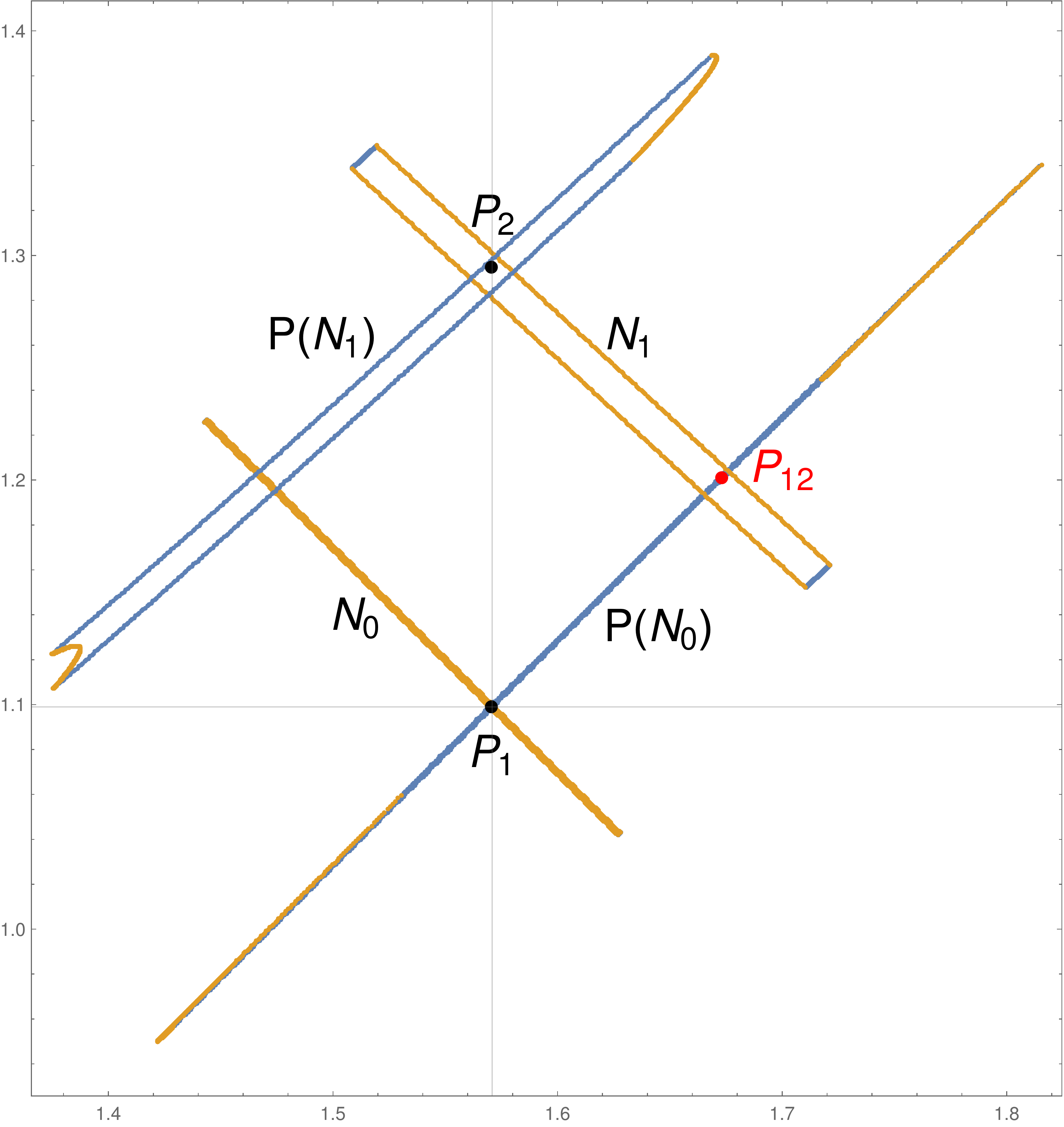}
	\caption{\label{fig:p12}The horseshoe proving symbolic dynamics for $P$, connecting the points $P_1$ and $P_2$. The exit sets of h-sets and their images are marked in orange.}
\end{figure}

\begin{theorem}
Let $N_0$ and $N_1$ be h-sets of the form $p + A \cdot b\cdot10^{-3}$, where:

\begin{center}
\begin{tabular}{|c||c|c|c|}
\hline
& $p$ & $A$ & $b$ \\
\hline
\hline
$N_0$ & $P_1$ & $M_1$ & $[-0.8,0.8]\times[-180,80]$\\
\hline
$N_1$ &$P_2$ & $M_2$ & $[-10,5]\times[-75,200]$\\
\hline
\end{tabular}
\end{center}


Then the following chain of covering relations occur:
\[
N_0 \overset{P}{\Longrightarrow} N_0 \overset{P}{\Longrightarrow} N_1 \overset{P}{\Longrightarrow} N_1 \overset{P}{\Longrightarrow} N_0\text{,}
\]
which proves the existence of symbolic dynamics for $P$.
\end{theorem}

\begin{proof}
Computer-assisted, \cite{proof}.
\end{proof}

\subsection{$P_1$ and $P_1$, $P_2$ and $P_2$}

To prove symbolic dynamics between $P_1$ and $P_1$ or between $P_2$ and $P_2$ with our simple tools, we need the second iteration of the Poincar\'e map $P$. Fig. \ref{fig:p11} illustrates the situation.

\begin{figure}[h]
	\includegraphics[height=7cm]{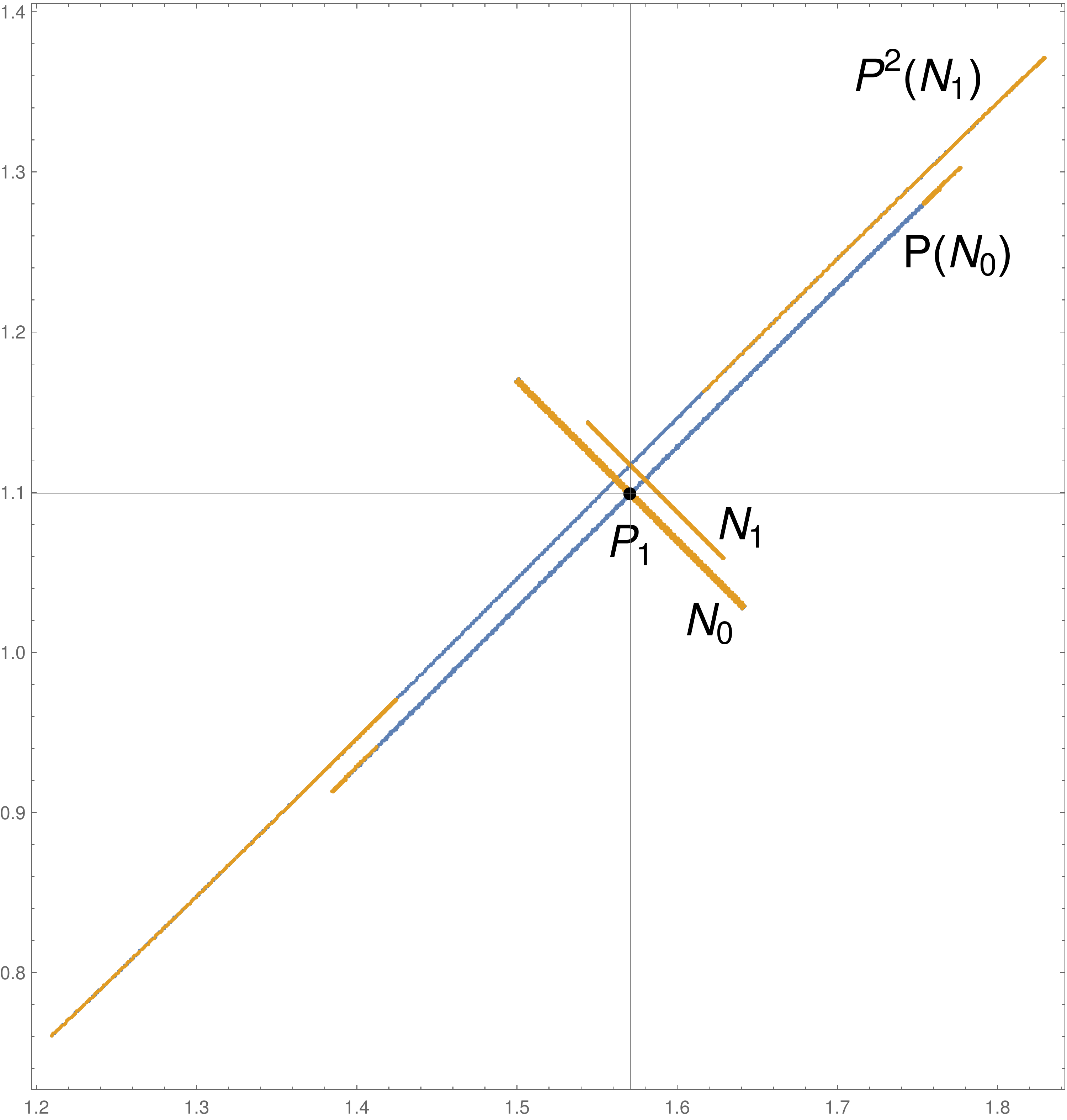}\hfil
	\includegraphics[height=7cm]{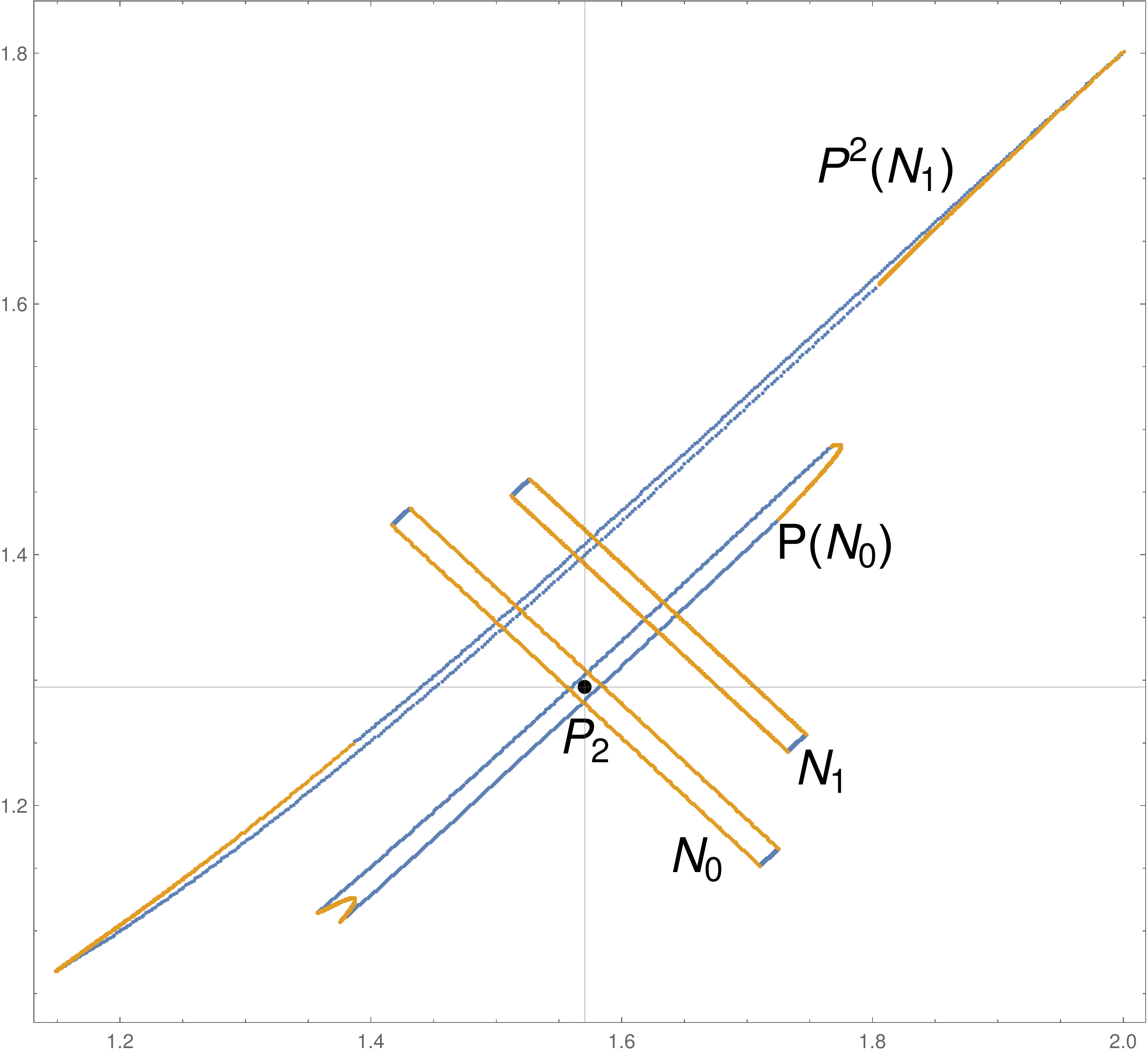}
	\caption{\label{fig:p11}The horseshoes proving symbolic dynamics for the second iteration of the Poincar\'e map $P^2$, connecting the point $P_1$ to itself (to the left) or $P_2$ to itself (to the right). The exit sets and their images are marked in orange.}
\end{figure}

\begin{theorem}
Let $N_0$ and $N_1$ be h-sets of the form $p + A \cdot b\cdot10^{-3}$, where:
\begin{center}
\begin{tabular}{|c||c|c|c|}
\hline
& $p$ & $A$ & $b$ \\
\hline
\hline
$N_0$ & $P_1$ & $M_1$ & $[-1,1]\times[-100,100]$\\
\hline
$N_1$ &$(1.58669; 1.10102)$ & $M_1$ & $[-0.1,0.1]\times[-60,60]$\\
\hline
\end{tabular}
\end{center}
or
\begin{center}
\begin{tabular}{|c||c|c|c|}
\hline
& $p$ & $A$ & $b$ \\
\hline
\hline
$N_0$ & $P_2$ & $M_2$ & $[-10,10]\times[-200,200]$\\
\hline
$N_1$ &$(1.62953; 1.35174)$ & $M_2$ & $[-10,10]\times[-150,150]$\\
\hline
\end{tabular}
\end{center}

Then, in both cases we have the following sequence of covering relations:
\[
N_0 \overset{P}{\Longrightarrow} N_0 \overset{P}{\Longrightarrow} N_1 \overset{P^2}{\Longrightarrow} N_1 \overset{P^2}{\Longrightarrow} N_0\text{,}
\]
which proves the existence of symbolic dynamics for $P^2$ (see Remark \ref{rem:chain}).
\end{theorem}

\begin{proof}
Computer-assisted, \cite{proof}.
\end{proof}

\subsection{$P_3$ and $P_3$}

To construct the horseshoe connecting $P_3$ to $P_3$, we will need five iterations of $P$ (see Fig. \ref{fig:p33}). Note that for $N_2$ and $N_3$ the direction matrices are corrected to make them compatible to the dynamics along the stable or unstable manifold.

\begin{figure}[h]
	\includegraphics[height=4cm]{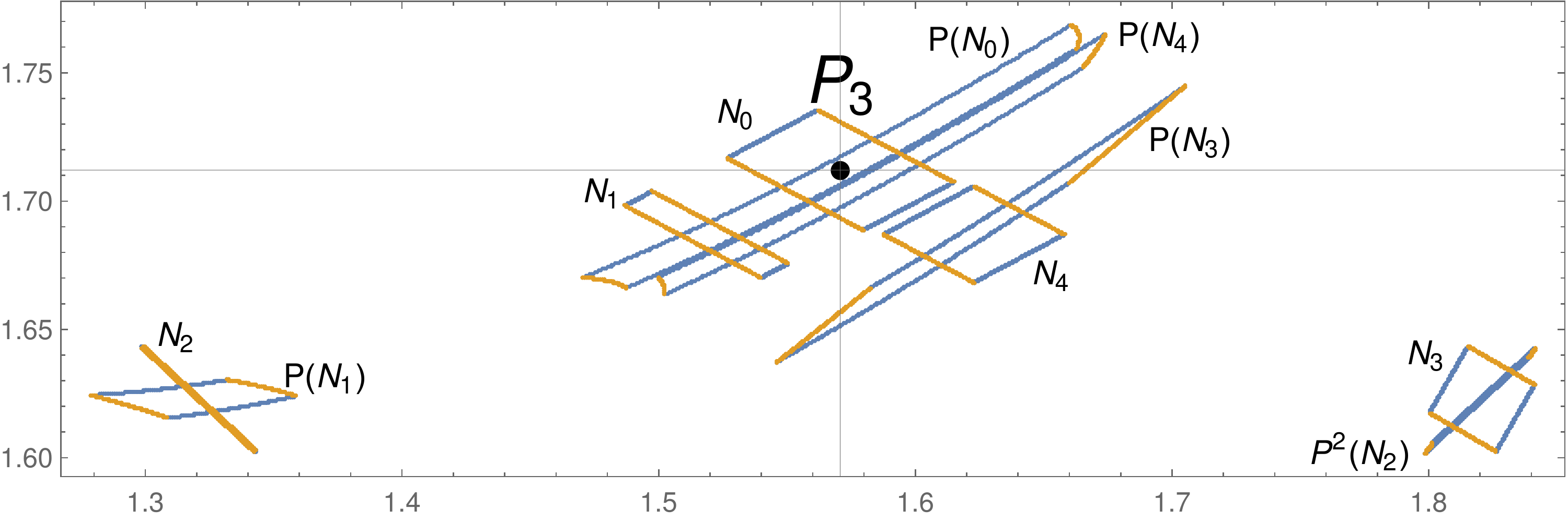}
	\caption{\label{fig:p33}The sequence of covering relations proving symbolic dynamics for $P^5$, connecting the point $P_3$ with itself. The exit sets and their images are marked in orange.}
\end{figure}

\begin{theorem}
Let $N_i$, $i=0,\ldots, 4$ be h-sets of the form $p + A \cdot b\cdot10^{-3}$, where:
\begin{center}
\begin{tabular}{|c||c|c|c|}
\hline
& $p$ & $A$ & $b$ \\
\hline
\hline
$N_0$ & $P_3$ & $M_3$ & $[-20,20]\times[-30,30]$\\
\hline
$N_1$ &$(1.51877;1.68699)$ & $M_3$ & $[-6,6]\times[-30,30]$\\
\hline
$N_2$ &$(1.32082;1.62293)$ & $\begin{bmatrix}
	0.734429 & 0.734429 \\
	0.678686 &-0.678686
\end{bmatrix}$ & $[-0.5,0.5]\times[-30,30]$\\
\hline
$N_3$ &$(1.82077;1.62293)$ & $\begin{bmatrix}	
	0.866025 & -0.5 \\ 0.5 & 0.866025
\end{bmatrix}$ & $[-15,15]\times[-15,15]$\\
\hline
$N_4$ &$(1.62282;1.68699)$ & $M_3$ & $[-20,20]\times[-20,20]$\\
\hline
\end{tabular}
\end{center}

Then we have the following sequence of covering relations:
\[
N_0 \overset{P}{\Longrightarrow} N_0 \overset{P}{\Longrightarrow} N_1 \overset{P}{\Longrightarrow} N_2 \overset{P^2}{\Longrightarrow} N_3 \overset{P}{\Longrightarrow} N_4  \overset{P}{\Longrightarrow} N_0
\text{, and also } N_4 \overset{P}{\Longrightarrow} N_1\text{,}
\]
which proves the existence of symbolic dynamics in $N_0\cup N_1$ for $P^5$ (see Remark \ref{rem:chain}).
\end{theorem}

\begin{proof}
Computer-assisted, \cite{proof}.
\end{proof}

\subsection{$P_1$ and $P_3$, $P_2$ and $P_3$}

To connect $P_1$ and $P_3$ we need the fourth iteration of $P$ (see Fig. \ref{fig:p13}), but this time we find only a one-way chain of covering relations from $P_1$ to $P_3$. 
Thanks to the time-reversing symmetry and the fact that we choose $N_0$ and $N_4$ symmetrical related to the $\theta = \frac{\pi}{2}$ line, we will be able to close this chain.

\begin{figure}[h]
	\includegraphics[height=10cm]{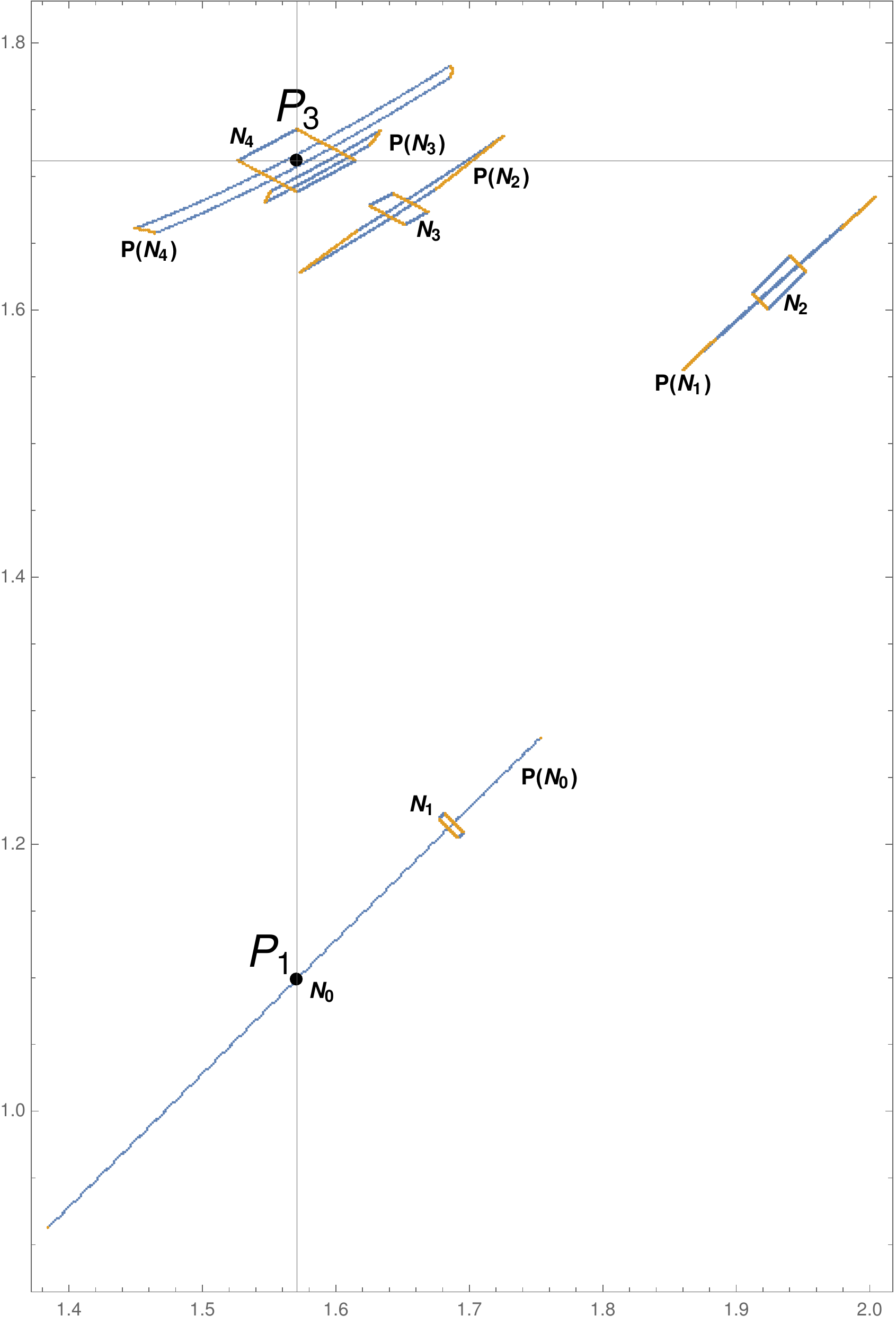}\hfil
	\includegraphics[height=10cm]{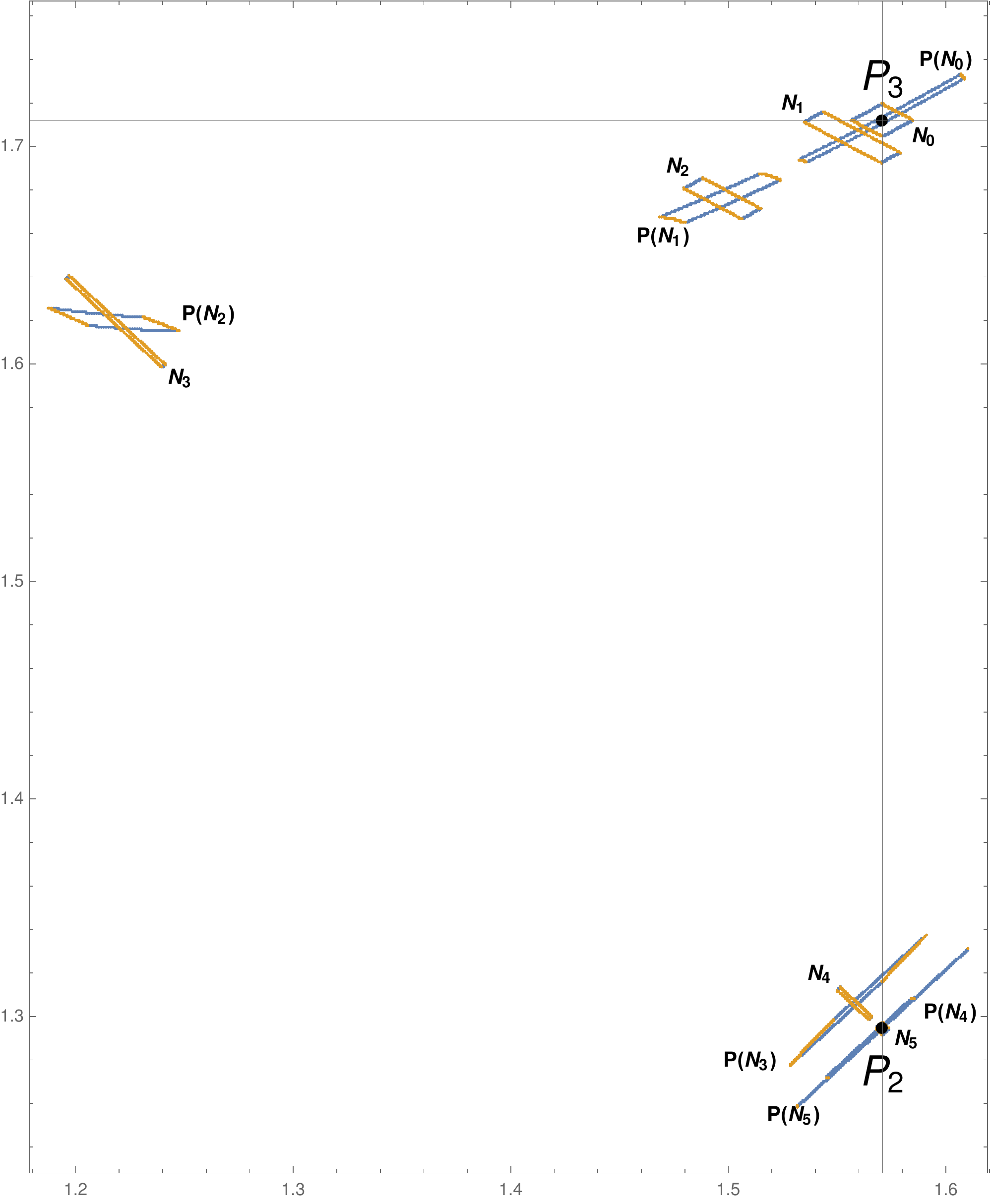}
	\caption{\label{fig:p13}\underline{To the left:} the right half of the horseshoe proving symbolic dynamics for $P^4$, connecting the points $P_1$ and $P_3$.
	\underline{To the right:} the left half of the horseshoe proving symbolic dynamics for $P^5$, connecting the points $P_3$ and $P_2$. The exit sets and their images are marked in orange.}
\end{figure}

\begin{theorem}\label{th:P13}
Let $N_i$, $i=0,\ldots,4$ be h-sets of the form $p + A \cdot b\cdot10^{-3}$, where:
\begin{center}
\begin{tabular}{|c||c|c|c|}
\hline
& $p$ & $A$ & $b$ \\
\hline
\hline
$N_0$ & $P_1$ & $M_1$ & $[-1,1]\times[-1,1]$\\
\hline
$N_1$ &$(1.68635;1.21391)$ & $M_1$ & $[-3,3]\times[-10,10]$\\
\hline
$N_2$ &$(1.93186;1.62049)$ & $M_1$ & $[-20,20]\times[-8,8]$\\
\hline
$N_3$ & $(1.6471;1.67581)$ & $M_3$ & $[-10,10]\times[-15,15]$\\
\hline
$N_4$ &$P_3$ & $M_3$ & $[-25,25]\times[-25,25]$\\
\hline
\end{tabular}
\end{center}

Then the following covering relations occur:
\begin{gather*}
N_0 \overset{P}{\Longrightarrow} N_0 \overset{P}{\Longrightarrow} N_1 \overset{P}{\Longrightarrow} N_2 \overset{P}{\Longrightarrow} N_3\overset{P}{\Longrightarrow} N_4 \overset{P}{\Longrightarrow} N_4
\text{.}
\end{gather*}
\end{theorem}

\begin{proof}
Computer-assisted, \cite{proof}.
\end{proof}

\begin{corollary}\label{cor:13}
Let $N_i$, $i=0,\ldots,4$ be h-sets defined in Theorem \ref{th:P13}. Then there is symbolic dynamics in $N_0 \cup N_4$ for $P^4$.
\end{corollary}

\begin{proof}
Consider the h-sets $R(N_i)^T$, $i=0,\ldots ,4$.
Using the time-reversing symmetry $R$ it is clear that
\begin{gather*}
R(N_0) \overset{P^{-1}}{\Longrightarrow} R(N_1) \overset{P^{-1}}{\Longrightarrow} R(N_2) \overset{P^{-1}}{\Longrightarrow} R(N_3)\overset{P^{-1}}{\Longrightarrow} R(N_4)
\text{,}
\end{gather*}
hence, from the definition of back-covering,
\begin{gather*}
R(N_4)^T \overset{P}{\Longleftarrow} R(N_3)^T \overset{P}{\Longleftarrow} R(N_2)^T \overset{P}{\Longleftarrow} R(N_1)^T\overset{P}{\Longleftarrow} R(N_0)^T
\text{.}
\end{gather*}

We have chosen $N_0$ and $N_4$ to be $\{\theta=\frac{\pi}{2}\}$-line symmetrical, so $R(N_0)^T=N_0$ and $R(N_4)^T=N_4$.
Therefore we get the full chain of covering and back-covering relations in the form:
\begin{gather*}
N_0 \overset{P}{\Rightarrow} N_0 \overset{P}{\Rightarrow} N_1 \overset{P}{\Rightarrow} N_2 \overset{P}{\Rightarrow} N_3\overset{P}{\Rightarrow}
 N_4 \overset{P}{\Rightarrow} N_4 \overset{P}{\Leftarrow} R(N_3)^T \overset{P}{\Leftarrow} R(N_2)^T \overset{P}{\Leftarrow} R(N_1)^T\overset{P}{\Leftarrow} N_0
\text{,}
\end{gather*}
which proves the existence of symbolic dynamics in $N_0\cup N_4$ for $P^4$, with the use of Remark \ref{rem:chain}.

\end{proof}

To connect $P_2$ and $P_3$ with symbolic dynamics, we need the fifth iteration of $P$ and time-reversing symmetry, see Fig. \ref{fig:p13} (right) for illustration.

\begin{theorem}\label{th:P23}
Let $N_i$, $i=0,\ldots,5$ be h-sets of the form $p + A \cdot b\cdot10^{-3}$, where:
\begin{center}
\begin{tabular}{|c||c|c|c|}
\hline
& $p$ & $A$ & $b$ \\
\hline
\hline
$N_0$ & $P_3$ & $M_3$ & $[-8,8]\times[-8,8]$\\
\hline
$N_1$ &$(1.5569;1.70419)$ & $M_3$ & $[-5,5]\times[-20,20]$\\
\hline
$N_2$ &$(1.49724;1.67606)$ & $M_3$ & $[-5,5]\times[-15,15]$\\
\hline
$N_3$ & $(1.21834;1.61946)$ & $M_2$ & $[-1,1]\times[-30,30]$\\
\hline
$N_4$ & $(1.55807;1.30592)$ & $M_2$ & $[-1,1]\times[-10,10]$\\
\hline
$N_5$ &$P_2$ & $M_2$ & $[-2,2]\times[-2,2]$\\
\hline
\end{tabular}
\end{center}

Then we have the following sequence of covering relations:
\begin{gather*}
N_0 \overset{P}{\Longrightarrow} N_0 \overset{P}{\Longrightarrow} N_1 \overset{P}{\Longrightarrow} N_2 \overset{P}{\Longrightarrow} N_3\overset{P}{\Longrightarrow} N_4 \overset{P}{\Longrightarrow} N_5\overset{P}{\Longrightarrow} N_5
\text{.}
\end{gather*}
\end{theorem}

\begin{proof}
Computer-assisted, \cite{proof}.
\end{proof}

\begin{corollary}
Let $N_i$, $i=0,\ldots,5$ be h-sets defined in Theorem \ref{th:P23}. Then there is symbolic dynamics in $N_0 \cup N_5$ for $P^5$.
\end{corollary}

\begin{proof}
Analogously as in the proof of Corollary \ref{cor:13}, the lacking h-sets on the right half-plane are simply $R(N_i)^T$, $i=0,\ldots ,5$.
Using the time-reversing symmetry and the definition of back-covering,
\begin{gather*}
R(N_5)^T \overset{P}{\Longleftarrow} R(N_4)^T \overset{P}{\Longleftarrow} R(N_3)^T \overset{P}{\Longleftarrow} R(N_2)^T \overset{P}{\Longleftarrow} R(N_1)^T\overset{P}{\Longleftarrow} R(N_0)^T
\text{.}
\end{gather*}

Again, we have chosen $N_0$ and $N_5$ to be $\{\theta=\frac{\pi}{2}\}$-line symmetrical, so $R(N_0)^T=N_0$ and $R(N_5)^T=N_5$.
Therefore we get the full chain of covering relations:
\begin{gather*}
N_0 \overset{P}{\Rightarrow} N_1 \overset{P}{\Rightarrow} N_2 \overset{P}{\Rightarrow} N_3\overset{P}{\Rightarrow}
 N_4 \overset{P}{\Rightarrow} N_5 \overset{P}{\Leftarrow} R(N_4)^T \overset{P}{\Leftarrow}  R(N_3)^T \overset{P}{\Leftarrow} R(N_2)^T \overset{P}{\Leftarrow} R(N_1)^T\overset{P}{\Leftarrow} N_0
\text{,}
\\
\text{and }\quad N_0 \overset{P}{\Rightarrow} N_0\text{, }\quad N_5 \overset{P}{\Rightarrow}N_5
\text{,}
\end{gather*}
which proves the existence of symbolic dynamics in $N_0\cup N_5$ for $P^5$.
\end{proof}




\bibliographystyle{plain}

\bibliography{Hyperion_bibliography}

\begin{thebibliography}{10}

\bibitem{proof}
{C}++ source code.
\newblock in preparation for online publication
  \texttt{http://ww2.ii.uj.edu.pl/\~{}zgliczyn/}.

\bibitem{B}
G.~J. Black, P.~D. Nicholson, and P.~C. Thomas.
\newblock Hyperion: Rotational dynamics.
\newblock {\em Icarus}, 117:149--161, 09 1995.

\bibitem{CAPD}
{CAPD group}.
\newblock {\em Computer Assisted Proofs in Dynamics C++ library}.
\newblock \texttt{http://capd.ii.uj.edu.pl}.

\bibitem{Danby}
J.~M.~A. Danby.
\newblock {\em Fundamentals of celestial mechanics}.
\newblock Macmillan, 1962.

\bibitem{Greiner}
W.~Greiner.
\newblock {\em Classical Mechanics: Systems of Particles and Hamiltonian
  Dynamics}.
\newblock Classical theoretical physics. Springer Berlin Heidelberg, 2009.

\bibitem{K}
J.~Jay~Klavetter.
\newblock Rotation of {H}yperion. {II} -- {D}ynamics.
\newblock {\em Astronomical Journal}, 98:1855--1874, 11 1989.

\bibitem{M}
R.~E. Moore.
\newblock {\em Interval analysis}.
\newblock Prentice-Hall series in automatic computation. Prentice-Hall, 1966.

\bibitem{Morse}
M.~Morse and G.~Hedlund.
\newblock Symbolic dynamics.
\newblock {\em Amer. J. Math.}, 60:815--866, 1938.

\bibitem{Moser}
J.~Moser.
\newblock {\em Stable and Random Motions in Dynamical Systems: With Special
  Emphasis on Celestial Mechanics (AM-77)}.
\newblock Princeton University Press, revised edition, 1973.

\bibitem{N}
A.~Neumaier.
\newblock {\em Interval Methods for Systems of Equations}.
\newblock Encyclopedia of Mathematics and its Applications. Cambridge
  University Press, 1991.

\bibitem{Tarnopolski2015}
M.~Tarnopolski.
\newblock Nonlinear time-series analysis of {H}yperion's lightcurves.
\newblock {\em Astrophysics and Space Science}, 357(2):160, May 2015.

\bibitem{Tarnopolski2016}
M.~Tarnopolski.
\newblock Influence of a second satellite on the rotational dynamics of an
  oblate moon.
\newblock {\em Celestial Mechanics and Dynamical Astronomy}, 07 2016.

\bibitem{WZ1}
D.~Wilczak and P.~Zgliczy\'nski.
\newblock Heteroclinic connections between periodic orbits in planar restricted
  circular three-body problem -- a computer assisted proof.
\newblock {\em Communications in Mathematical Physics}, 234(1):37--75, Mar
  2003.

\bibitem{WZ2}
D.~Wilczak and P.~Zgliczy{\'{n}}ski.
\newblock Heteroclinic connections between periodic orbits in planar restricted
  circular three body problem. {P}art {II}.
\newblock {\em Communications in Mathematical Physics}, 261(2):547--547, Jan
  2006.

\bibitem{W}
J.~Wisdom, S.~J. Peale, and F.~Mignard.
\newblock The chaotic rotation of {H}yperion.
\newblock {\em Icarus}, 58(2):137--152, 1984.

\bibitem{GZ}
P.~Zgliczy\'nski and M.~Gidea.
\newblock Covering relations for multidimensional dynamical systems.
\newblock {\em Journal of Differential Equations}, 202(1):32--58, 2004.

\end{thebibliography}

\end{document}